\documentclass[11pt]{article}
\usepackage{mathrsfs}
\usepackage{amssymb}
\usepackage{amsmath,amsthm,amssymb,mathrsfs,amsbsy,bm}
\usepackage{graphicx}
\usepackage{epsfig}
\usepackage{enumerate}
\usepackage{color, xcolor} 

\usepackage[colorlinks=true, allcolors=blue]{hyperref}

\usepackage[letterpaper,top=2cm,bottom=2cm,left=3cm,right=3cm,marginparwidth=1.75cm]{geometry}

\numberwithin{equation}{section}

\newtheorem{theorem}{\bf{Theorem}}[section]
\newtheorem{lemma}{\bf {Lemma}}[section]
\newtheorem{define}{\bf{Definition}}[section]

\newtheorem{remark}{\bf{Remark}}[section]

\newcommand{\mbE}{\widehat{\mathbb{E}}}
\newcommand{\mbe}{\widehat{\mathcal{E}}}
\newcommand{\V}{\mathbb{V}}
\newcommand{\mv}{\mathcal{V}}
\newcommand{\sles}{(\Omega,\mathcal{H},\mbE)}
\newcommand{\um}{\overline{\mu}}
\newcommand{\lm}{\underline{\mu}}
\newcommand*{\dif}{\mathop{}\!\mathrm{d}}
\newcommand{\Vstar}{\widehat{\mathbb{V}}^*}
\newcommand{\mvstar}{\widehat{\mathcal{V}}^*}
\newcommand{\hV}{\widehat{\mathbb{V}}}
\newcommand{\mhV}{\widehat{\mathcal{V}}}
\newcommand{\bE}{\breve{\mathbb{E}}}
\newcommand{\be}{\breve{\mathcal{E}}}

\title{Strong law of large numbers for $m$-dependent and stationary random variables under sub-linear expectations\thanks{This work was supported by grants from the NSF of China  (Grant Nos. U23A2064 and 12031005)} \footnote{This paper is submitted to a special issue of Science in China-Mathematics (Chinese)  to congratulate  Professor Lin Zhengyan on his 85th birthday}}
\author{Wang-Yun Gu, Li-Xin Zhang\footnote{Corresponding author, email:stazlx@zju.edu.cn }}
\date{}

\begin{document}
\maketitle

\begin{abstract}
\par The arm of this paper is to establish  the strong law of large numbers (SLLN) of $m$-dependent random variables under the framework of sub-linear expectations. We establish the SLLN for a sequence of independent, but not necessarily identically distributed random variables. The study further extends the SLLN to $m$-dependent and stationary sequence of random variables with the condition $C_{\hV}(|X_1|)<\infty$ which is the sufficient and necessary condition of SLLN in the case of independent and identically distributed random variables.
\end{abstract}

\section{Introduction}
\par The classical law of large numbers (LLN) stands as a cornerstone of probability theory, underpinning the predictable behavior of sample averages as the sample size grows. The well-known Kolmogorov\cite{Kol30}'s strong law of large numbers (SLLN) typically assumes sequences of independent and identically distributed (i.i.d.) random variables.
However, this assumption does not account for the complexities of real-world data, which often involve non-identical distributions and dependencies among observations. The laws of large numbers as well as the central limit theorems for independent random variables have been extended for various kinds of dependent random variables.
Since Professor Lin Zhengyan published his first paper\cite{Lzy1965} on the limit theory of dependent variables in 1965, he has been conducting decades of research on various weak and strong limit theorems of dependent variables and has published over 90 related papers (c.f. https://mathscinet.ams.org/mathscinet/author?authorId=191973). The monograph\cite {LzyLcr} published in 1997 summarized Lin's achievements in mixing dependent variables. 
As a direct extension of independence, $m$-dependence is a common weakening condition.  In 1980 and 1981, Lin Zhengyan\cite{LZY1980, LZY1981} proved the central limit theorem and weak invariance principle for $m$-dependent and stationary random variables without finite variances.  Compared to the central limit theorem, in the classical probability framework, due to the additivity of expectations and probabilities, the law of large numbers for a sequence of $m$-dependent random variables can be easily derived from the results for independent variables. Recent works on this topic can be found in Zhang \cite {zhang98} and Thanh \cite {Thanh05} who showed the strong laws of large numbers under the block-wise and pair-wise $m$-dependent settings. 

\par Sub-linear expectations, introduced by Peng Shige, offer a framework for handling uncertainties by relaxing the linearity and additivity conditions of classical expectations. Peng\cite{Peng08} proved the weak law of large numbers and the central limit theorem under his framework of sub-linear expectations. Not long after, Professor Lin Zhengyan organized seminars at Zhejiang University on the theory of sub-linear expectations and vigorously suggested studying the limit theorems under the framework of sub-linear expectations. Lin and Zhang\cite {LZ2017} proved that the self-normalized partial sum process of i.i.d. random variables weakly converges to a $G$-Brownian motion self-normalized by its quadratic variation, and obtained the self-normalized central limit theorem with limit distributions quite different from those in classical probability theory. Chen\cite{chen16} and Zhang\cite{zhang16} obtained the strong law of large numbers for i.i.d. random variables under the sub-linear expectation. However, to obtain the precise limits and limit points, Chen\cite{chen16} and Zhang\cite{zhang16} assumed that the capacity related to the sub-linear expectation is continuous. Zhang\cite{zhang21} pointed out that the continuity of the capacity is a very strict condition, and showed that if a capacity related to the sub-linear expectation is continuous and there is a sequence $\{X_i;i\ge 1\}$ of i.i.d. random variables in the sub-linear expectation space, then the sub-linear expectation is a linear expectation on the space of $\sigma(X_i;i\ge 1)$-measurable functions.  
In recent works,  Song \cite {Song2022} obtained a strong law of large numbers for i.i.d. random variables relative to a sub-linear expectation represented by a weakly compact family of probability measures on a Polish space under the $(1+\alpha)$th moment condition, and, Zhang\cite{zhang23b} gave a reasonable condition on the general sub-linear expectation and the capacity for the strong limit laws, and established the sufficient and necessary conditions of strong law of large numbers for i.i.d. random variables. However, the limit theorems for the dependent random variables under the sub-linear expectation are really rare.   Lin\cite{lin22} proved a weak LLN for $m$-dependent random variables with the method of Li\cite{Li15}. This paper aims to extend the strong law of large numbers into this broader context. Due to the nonadditivity of sub-linear expectations and the discontinuity of capacities, the most basic law of large numbers cannot be directly obtained from the results for i.i.d. random variables in the case of $m$-dependence. In this paper, we first establish a general law of large numbers for independent but non-identically distributed sequences under sub-linear expectations, thereby capturing a wider array of possible scenarios. Based on this, we borrow the technique of segmenting summation to piecewise summations of large and small blocks used by Lin Zhengyan \cite{Lzy1965, LZY1981} to obtain the strong law of large numbers for $m$-dependent and linearly stationary random variables under conditions as weak as those in the i.i.d case. Meanwhile, we also proved that even for $m$-dependent and identically distributed random variables, the existence of Choquet's expectation is still a necessary condition for the strong law of numbers to hold.

\par The remainder of this paper is structured as follows. In Section \ref{sec:setting}, we briefly introduce the framework of sub-linear expectations.  Our main theorems are presented in Section \ref{sec:result} and their detailed proofs are given in Section \ref{sec:proof}. In  Section \ref{sec:inequality}, we list the inequalities needed in our proofs.

\section{Basic settings}\label{sec:setting}
\par We use the framework and notations of Peng\cite{Peng08, Peng09, Peng19}. If one is familiar with these notations, he or she can skip this section. Let $(\Omega,\mathcal{F})$ be a given measurable space and let $\mathcal{H}$ be a linear space of real functions defined on  $(\Omega,\mathcal{F})$ such that if $X_1,\cdots, X_n\in\mathcal{H}$, then $\varphi(X_1,\cdots,X_n)\in\mathcal{H}$ for each $\varphi\in C_{l,Lip}(\mathbb{R}^n)$, where $C_{l,Lip}(\mathbb{R}^n)$ denotes the linear space of local Lipschitz functions $\varphi$ satisfying
\begin{align*}
\vert \varphi(\bm{x})-\varphi(\bm{y})& \vert\leq C(1+\vert \bm{x}\vert^m+\vert\bm{y}\vert^m)\vert\bm{x}-\bm{y}\vert,\quad\forall\bm{x},\bm{y}\in\mathbb{R}^n,\\
 & \text{ for some } C>0, m\in\mathbb{N} \text{  depending on } \varphi.
 \end{align*}
$\mathcal{H}$ is considered as a space of "random variables". We also denote $C_{b, Lip}(\mathbb{R}^n)$ the space of bounded Lipschitz functions. In this case, we denote $X\in\mathcal{H}$.
\begin{define}
A sub-linear expectation $\mbE$ on $\mathcal{H}$ is a function $\mbE:\mathcal{H}\rightarrow\bar{\mathbb{R}}$ satisfying the following properties: for all $X,Y\in\mathcal{H}$, we have
\begin{itemize}
\item[(a)] Monotonicity: If $X\geq Y$, then $\mbE[X]\geq\mbE[Y]$;
\item[(b)] Constant preserving: $\mbE[c]=c$;
\item[(c)] Sub-additivity: $\mbE[X+Y]\leq\mbE[X]+\mbE[Y]$ whenever $\mbE[X]+\mbE[Y]$ is not of the form $+\infty-\infty$ or $-\infty+\infty$;
\item[(d)] Positive homogeneity: $\mbE[\lambda X]=\lambda\mbE[X]$ for $\lambda>0$.
\end{itemize}
Here, $\bar{\mathbb{R}}=[-\infty,\infty]$, $0\cdot\infty$ is defined to be 0. The triple $\sles$ is called a sub-linear expectation space. Give a sub-linear expectation $\mbE$, let us denote the conjugate expectation $\mbe$ of $\mbE$ by
$$
\mbe[X]:=-\mbE[-X],\quad\forall X\in\mathcal{H}.
$$
\end{define}
\par From the definition, it is easily shown that $\mbe[X]\leq\mbE[X],\enspace \mbE[X+c]=\mbE[X]+c$, and $\mbE[X-Y]\geq\mbE[X]-\mbE[Y]$ for all $X,Y\in\mathcal{H}$ with $\mbE[Y]$ being finite. We also call $\mbE[X]$ and $\mbe[X]$ the upper-expectation and lower-expectation of $X$, respectively.

\begin{define}
\begin{itemize}
\item[(i)] (Identical distribution) Let $\bm{X}_1$ and $\bm{X}_2$ be two n-dimensional random vectors, respectively, defined in sub-linear expectation spaces $(\Omega_1,\mathcal{H}_1,\mbE_1)$ and $(\Omega_2,\mathcal{H}_2,\mbE_2)$. They are called identically distributed, denoted by $\bm{X}_1\overset{d}{=}\bm{X}_2$, if
$$
	\mbE_1[\varphi(\bm{X}_1)]=\mbE_2[\varphi(\bm{X}_2)],\quad\forall\varphi\in C_{b,Lip}(\mathbb{R}^n).
$$
 A sequence $\{X_n;n\geq 1\}$ of random variables is said to be identically distributed if $X_i\overset{d}{=}X_1$ for each $i\geq 1$.
\item[(ii)] (Independence) In a sub-linear expectation space $\sles$, a random vector $\bm{Y}=(Y_1,\cdots,Y_n),Y_i\in\mathcal{H}$ is said to be independent of another random vector $\bm{X}=(X_1,\cdots,X_m),X_i\in\mathcal{H}$ under $\mbE$ if for each test function $\varphi\in C_{b,Lip}(\mathbb{R}^m\times\mathbb{R}^n)$ we have $\mbE[\varphi(\bm{X},\bm{Y})]=\mbE[\mbE[\varphi(\bm{x},\bm{Y})]\vert_{\bm{x}=\bm{X}}]$.
\item[(iii)] (Independent random variables) A sequence of random variables $\{X_n;n\geq 1\}$ is said to be independent if $X_{i+1}$ is independent of $(X_1,\cdots,X_i)$ for each $i\geq 1$.
\item[(iv)](m-dependence) A sequence of random variables(or random vectors) $\{X_n;n\geq 1\}$ is said to be $m$-dependent if there exists an integer $m$ such that for every $n$ and every $j\geq m+1$, $(X_{n+m+1},\cdots,X_{n+j})$ is independent of $(X_1,\cdots,X_n)$. In particular, if $m=0$, $\{X_n;n\geq 1\}$ is an independent sequence.
\item[(v)](stationary) A sequence of random variables(or random vectors) $\{X_n;n\geq 1\}$ is said to be stationary if for every positive integers $n$ and $p$, $(X_1,\cdots,X_n)\overset{d}{=}(X_{1+p},\cdots,X_{n+p})$.
\item[(vi)](linear stationary) A sequence of random variables(or random vectors) $\{X_n;n\geq 1\}$ is said to be linear stationary if for every positive integers $n$ and $p$, $X_1+\cdots+X_n\overset{d}{=}X_{1+p}+\cdots+X_{n+p}$.
\end{itemize}
\end{define}
It is easily seen that if $\{X_1,\cdots,X_n\}$ are independent and bounded random variables, then $\mbE[\sum_{i=1}^nX_i]=\sum_{i=1}^n\mbE[X_i]$.
\par Next, we consider the capacities corresponding to the sub-linear expectations. Let $\mathcal{G}\subset\mathcal{F}$. A function $V:\mathcal{G}\rightarrow[0,1]$ is called a capacity if
$$
V(\emptyset)=0,\enspace V(\Omega)=1 \enspace and\enspace  V(A)\leq V(B)\enspace \forall A\subset B, A,B\in\mathcal{G}.
$$
It is called sub-additive if $V(A\cup B)\leq V(A)+V(B)$ for all $A,B\in\mathcal{G}$ with $A\cup B\in\mathcal{G}$.

\par Let $\sles$ be a sub-linear expectation space. We define $(\V,\mv)$ as a pair of capacities with the properties that
\begin{equation}
	\mbE[f]\leq\V(A)\leq\mbE[g]\quad if\enspace f\leq I_A\leq g,f,g\in\mathcal{H}\enspace and\enspace A\in\mathcal{F},\label{eq1.1}
\end{equation}
$\V$ is sub-additive and $\mv(A):=1-\V(A^c),A\in\mathcal{F}$.
We call $\V$ and $\mv$ the upper and lower capacity, respectively.
\par Also, we define the Choquet integrals/expectations $(C_{\V},C_{\mv})$ by
\begin{equation}
	C_V[X]=\int_0^{\infty}V(X\geq t)\dif t+\int_{-\infty}^0[V(X\geq t)-1]\dif t
\end{equation}
with $V$ being replaced by $\V$ and $\mv$, respectively.  If $\V$ on the sub-linear expectation space $\sles$ and $\widetilde{\mathbb{V}}$ on the sub-linear expectation space $(\widetilde{\Omega},\widetilde{\mathcal{H}},\widetilde{\mathbb{E}})$ are two capacities have the property \eqref{eq1.1}, then for any random variables $X\in\mathcal{H}$ and $\widetilde{X}\in\widetilde{\mathcal{H}}$ with $X\overset{d}{=}\widetilde{X}$, we have
\begin{equation}
	\V(X\geq x+\epsilon)\leq\widetilde{V}(\widetilde{X}\geq x)\leq\V(X\geq x-\epsilon)\quad for\enspace all\enspace\epsilon>0\enspace and \enspace x\label{eq1.2}
\end{equation}
and
\begin{equation}
C_{\V}[X]=C_{\widetilde{\V}}[X].
\end{equation}
\par In general, we choose $(\V,\mv)$ as
\begin{equation}
\hat{\V}(A):=\inf\{\mbE[\xi]:I_A\leq\xi,\xi\in\mathcal{H}\},\widehat{\mv}(A)=1-\widehat{\V}(A^c), \enspace\forall A\in\mathcal{F}.
\end{equation}
\par Since $\widehat{\V}$ may be not countably sub-additive so that the Borel-Cantelli lemma is not valid, we consider its countably sub-additive extension $\Vstar$ which defined by
\begin{equation}
\Vstar(A):=\inf\left\{\sum_{n=1}^{\infty}\widehat{\V}(A_n):A\subset\bigcup_{n=1}^{\infty}A_n\right\},\mvstar(A)=1-\Vstar(A^c),\quad A\in\mathcal{F}.
\end{equation}
As shown in Zhang\cite{zhang16}, $\Vstar$ is countably sub-additive, and $\Vstar(A)\leq\widehat{\V}(A)$. Further, $\widehat{\V}(A)$(resp. $\Vstar$) is the largest sub-additive(resp. countably sub-additive) capacity in sense that if $V$ is also a sub-additive(resp. countably sub-additive) capacity satisfying $V(A)\leq\mbE[g]$ whenenver $I_A\leq g\in\mathcal{H}$, then $V(A)\leq\widehat{\V}(A)$(resp. $V(A)\leq\Vstar(A)$).

\par Finally, we introduce the condition (CC) proposed in Zhang\cite{zhang21}. We say that the sub-linear expectation $\mbE$ satisfies the condition (CC) if
$$
		\mbE[X]=\sup_{P\in\mathcal{P}}P[X],\enspace X\in\mathcal{H}_b,
$$
	where $\mathcal{H}_b=\{f\in\mathcal{H}:f\enspace is\enspace bounded\}$, $\mathcal{P}$ is a countable-dimensionally weakly compact family of probability measures on $(\Omega,\sigma(\mathcal{H}))$ in sense that, for any $Y_1, Y_2,\cdots\in\mathcal{H}_b$ and any sequence $\{P_n\}\subset\mathcal{P}$, there is a subsequence $\{n_k\}$ and a probability measure $P\in\mathcal{P}$ for which
$$
	\lim_{k\rightarrow\infty}P_{n_k}[\varphi(Y_1,\cdots,Y_d)]=P[\varphi(Y_1,\cdots,Y_d)],\enspace \varphi\in C_{b,Lip}(\mathbb{R}^d),d\geq1.
$$
We denote
\begin{equation}
	\V^{\mathcal{P}}(A)=\sup_{P\in\mathcal{P}}P(A),\enspace A\in\sigma(\mathcal{H}),
\end{equation}
and it is obvious that $\V^{\mathcal{P}}\leq\widehat{\V}^*\leq \hV$. Let
\begin{align*}
 \mathcal{P}^e=& \left\{P: P \text{ is a probability measure on }\sigma(\mathcal{H}) \right. \\
 & \left.\text{ such that }
P[X]\le \mbE[X] \text{ for all } X\in \mathcal{H}_b \right\}.
\end{align*}
Zhang (2023) has shown the following three statements are equivalent: (i) the condition (CC) is satisfied with some $\mathcal{P}$; (ii) the condition (CC) is satisfied with  $\mathcal{P}^e$; (iii) $\mbE$ is regular in the sense that $\mbE[X_n]\searrow 0$ whenever $\mathcal{H}_b\ni X_n\searrow 0$. It shall be mentioned that  $\sup_{P\in\mathcal{P}_1}P[X] =\sup_{P\in\mathcal{P}_2}P[X]\; \forall  X\in\mathcal{H}_b$ does not imply $\sup_{P\in\mathcal{P}_1}P(A) =\sup_{P\in\mathcal{P}_2}P(A)\; \forall  A\in\sigma(\mathcal{H})$.

\par Through this paper, for real numbers $x$ and $y$, we denote $x\vee y=\max\{x,y\},x\wedge y=\min\{x,y\},x^+=x\vee 0$ and $x^-=x\wedge 0$. For a random variable $X$, because $XI\{\vert X\vert\leq c\}$ may not be in $\mathcal{H}$, we will truncate it in the form $(-c)\vee X\wedge c$ denoted by $X^{(c)}$. We define $\bE[X]=\lim_{c\rightarrow\infty}\mbE[X^{(c)}]$ if the limit exists, and $\be[X]=-\bE[-X]$. It is obvious that, if $C_{\widehat{\V}}(|X|)<\infty$, then
$\bE[X]$, $\be[X]$ and $\bE[|X|]$, and $ \bE[X], \be[X]\le \bE[|X|]\le C_{\widehat{\V}}(|X|)$.

\section{Main results}\label{sec:result}
Our first theorem shows a strong limit theorem for a sequence of independent, but not necessarily identically distributed random variables under the sub-linear expectation.
\begin{theorem}\label{th1}
	Let $\{X_n;n\geq 1\}$ be a sequence of independent random variables in the sub-linear expectation space $\sles$. Suppose   $\{a_n\}$ is a sequence such that $1\leq a_n\nearrow\infty$ and
\begin{equation}
	\sum_{n=1}^{\infty}\frac{\mbE[X_n^2]}{a_n^2}<\infty.\label{eqth1.1}
\end{equation}
Denote $S_n=\sum_{i=1}^nX_i$. Then
$$
	\Vstar\left(\limsup_{n\rightarrow\infty}\frac{S_n-\mbE[S_n]}{a_n}>0\enspace or\enspace\liminf_{n\rightarrow\infty}\frac{S_n-\mbe[S_n]}{a_n}<0\right)=0.
$$
\end{theorem}

Under the condition (CC), the limit theorems for some probability measure $P\in\mathcal{P}$ can be obtained.
\begin{theorem}\label{th2}
	Let $\{X_n\}$ and $\{a_n\}$ satisfy the same conditions as in Theorem \ref{th1} and we further assume that $\mbE$ satisfies the condition (CC). Then the following two conclusions hold.
\begin{itemize}
	\item[(i)] There exists $P\in\mathcal{P}$ such that
		\begin{equation}
			P\left(\limsup_{n\rightarrow\infty}\frac{S_n-\mbE[S_n]}{a_n}=0\enspace and\enspace\liminf_{n\rightarrow\infty}\frac{S_n-\mbe[S_n]}{a_n}=0\right)=1.\label{eqth2.1}
	\end{equation}
	\item[(ii)]	For any sequence $\{\mu_n\}$ with $\mbe[X_n]\leq\mu_n\leq\mbE[X_n]$, there exists $P\in\mathcal{P}$ such that
	\begin{equation}
			P\left(\lim_{n\rightarrow\infty}\frac{\sum_{i=1}^n(X_i-\mu_i)}{a_n}=0\right)=1.\label{eqth2.2}
	\end{equation}
\end{itemize}
\end{theorem}
With the previous preparations, now we focus on $m$-dependent case. The following theorem is the law of large numbers for $m$-dependent and stationary random variables.
\begin{theorem}\label{th3}
	Let $\{X_n;n\geq 1\}$ be a sequence of $m$-dependent and linear stationary random variables in the sub-linear expectation space $\sles$ with
\begin{equation}\label{eqth3.1}
	C_{\hV}(|X_1|)<\infty.
\end{equation}
Denote $S_n=\sum_{i=1}^n X_i$.
Then there exist real numbers $\lm,\um$ such that $\be[X_1]\leq\lm\leq\um\leq\bE[X_1]$, and
\begin{equation}
	 \um=\lim_{n\rightarrow\infty}\frac{\bE[S_n]}{n},\enspace\lm=\lim_{n\rightarrow\infty}\frac{\be[S_n]}{n},\quad\quad\label{eqth3.2}
\end{equation}
\begin{equation}	 \Vstar\left(\limsup_{n\rightarrow\infty}\frac{S_n}{n}>\um\enspace or\enspace\liminf_{n\rightarrow\infty}\frac{S_n}{n}<\lm\right)=0.\label{eqth3.3}
\end{equation}
Further, suppose $\mbE$ satisfies the condition (CC), then there exists $P\in\mathcal{P}$ such that
\begin{equation}
	P\left(\limsup_{n\rightarrow\infty}\frac{S_n}{n}=\um\enspace and\enspace\liminf_{n\rightarrow\infty}\frac{S_n}{n}=\lm\right)=1,
\label{eqth3.4}
\end{equation}
 for any sequence $\{\mu_j;j\ge 1\}$ with $\lm\leq\mu_j\leq\um$, there exists $P\in\mathcal{P}$ such that
\begin{equation}
	P\left(\lim_{n\rightarrow\infty}\frac{S_n-\sum_{j=1}^n \mu_j}{n}=0\right)=1,
\label{eqth3.5}
\end{equation}
for any $a,b$ with $\lm\leq a\leq b\leq\um$, there exists $P\in\mathcal{P}$ such that
\begin{equation}
	P\left(C\Big\{\frac{S_n}{n}\Big\}=[a,b]\right)=1,
\label{eqth3.6}
\end{equation}
where $C\{x_n\}$ denotes the cluster set of a sequence of $\{x_n\}$ in $\mathbb R$.
\end{theorem}

\begin{remark}
	 \eqref{eqth3.4}  and  \eqref{eqth3.5}  imply  that for $\V=\V^{\mathcal{P}}$, $\Vstar$ or $\widehat{\V}$,
$$
	\V\left(\limsup_{n\rightarrow\infty}\frac{S_n}{n}=\um\enspace and\enspace\liminf_{n\rightarrow\infty}\frac{S_n}{n}=\lm\right)=1,
$$
and for any $\lm\leq\mu\leq\um$,
$$
	\V\left(\lim_{n\rightarrow\infty}\frac{S_n}{n}=\mu\right)=1.
$$
\end{remark}

The next theorem shows that the condition \eqref{eqth3.1} is a necessary condition of the strong law of large numbers.
\begin{theorem}\label{th4}
	Let $\{X_n;n\geq 1\}$ be a sequence of $m$-dependent and identically distributed  random variables in the sub-linear expectation space $\sles$  satisfying the condition (CC). If
\begin{equation}\label{eqth4.1}
	C_{\hV}(|X_1|)=\infty,
\end{equation}
 then exists a probability measure $P\in\mathcal{P}$ such that
 \begin{equation} \label{eqth4.2}
 \Vstar\left(\limsup_{n\rightarrow\infty}\frac{|S_n|}{n}=\infty\right)= P\left(\limsup_{n\rightarrow\infty}\frac{|S_n|}{n}=\infty\right)=1.
 \end{equation}

\end{theorem}

\section{Related inequalities and properties}\label{sec:inequality}
Before we strat to show our main results, we need some basic lemmas as preparation. 
The first two lemmas are about 
exponential inequalities and Kolmogorov's maximal inequalities under both $\V$ and $\mv$, whose proofs are given in Lemma 3.1 of Zhang\cite{zhang21} and Lemma 2.7 of Zhang\cite{zhang23b}, are also needed in our proof.
\begin{lemma}
 Let $\{X_1,\cdots,X_n\}$ be a sequence of independent random variables in the sub-linear expectation space $\sles$. Set $S_n=\sum_{i=1}^n X_i, B_n^2=\sum_{i=1}^n \mbE[X_i^2]$. Then for all $x,y>0$, $0<\delta\leq1$ and $p\ge 2$,
\begin{align}
	 &\V\left(\max_{k\leq n}(S_k-\mbE[S_k])\geq x\right)\quad\left(resp.\mv\left(\max_{k\leq n}(S_k-\mbe[S_k])\geq x\right)\right)\nonumber\\
 \leq&C_p\delta^{-p}x^{-p}\sum_{i=1}^n\mbE[X_i^2]+\exp\left\{-\frac{x^2}{2(1+\delta)B_n^2}\right\}.
\end{align}
Further, by noting that $xe^{-x}\leq e^{-1}$ when $x\geq0$, we have Kolmogorov's maximal inequalities under both $\V$ and $\mv$ as follows:
\begin{align}
	&\V\left(\max_{k\leq n}(S_k-\mbE[S_k])\geq x\right)\leq Cx^{-2}\sum_{i=1}^n\mbE[X_i^2],\label{kolV}\\
	&\mv\left(\max_{k\leq n}(S_k-\mbe[S_k])\geq x\right)\leq Cx^{-2}\sum_{i=1}^n\mbE[X_i^2].\label{kolv}
\end{align}
\label{Kol}
\end{lemma}
\begin{lemma}
	Let $\{X_{k};k=1,\cdots,n\}$ be a sequence of independent random variables in the sub-linear expectation space $\sles$ such that $\mbE[X_{k}^2]<\infty,k=1,\cdots,n$. Then for any constants $\{\mu_k\}$ satisfying $\mbe[X_k]\leq\mu_{k}\leq\mbE[X_k],k=1\cdots,n$, we have
$$
	\mv\left(\max_{k\leq n}\left|\sum_{i=1}^k(X_i-\mu_i)\right|\geq x\right)\leq \frac{2}{x^2}\sum_{i=1}^n\mbE[X_i^2],\enspace\forall x>0.
$$
\label{lemma2.7}
\end{lemma}

The next lemma is "the convergent part" of Borel-Cantelli lemma under the sub-linear expectation space, which have been proved in many papers.
\begin{lemma}\label{lemBC1}
(\cite{zhang21} Lemma 4.1) (i) Let $\{A_n,n\geq 1\}$ be a sequence of events in $\mathcal{F}$. Suppose that $\V$ is a sub-additive capacity and $\sum_{n=1}^\infty \V(A_n)<\infty$. Then
	$$
	\lim_{n\rightarrow\infty}\max_N\V\left(\bigcup_{i=n}^N A_i\right)=0.
$$
	If $\V$ is a countably sub-additive capacity, then
	$$
	\V(A_n,i.o.)=0.
	$$
\label{BC}
\end{lemma}

The next lemma on the converse part of the Borel-Cantelli lemma is a refinement of Lemma 2.3 of Zhang\cite{zhang23b}.

\begin{lemma}\label{lemBC2} Let $(\Omega,\mathcal{H},\mbE)$ be a sub-linear expectation  space with  a capacity $\V$ having  the property \eqref{eq1.1}, and $\mv=1-\V$. Suppose the condition (CC) is satisfied. 
 \begin{description}
   \item[\rm (i)] Let $\{X_n;n\ge 1\}$ be a sequence of independent random variables in  $(\Omega,\mathcal{H},\mbE)$. If $\sum_{n=1}^{\infty}\mv(X_n<1)<\infty$,  then there exists $P\in\mathcal{P}$ such that 
   \begin{equation}\label{eqlemBC2.1}P\left(\bigcup_{m=1}^{\infty} \bigcap_{i=m}^n \{X_i\ge 1\}\right)=1\;\; i.e., \; P\left( X_i<1 \; i.o.\right)=0.
   \end{equation}
   \item[\rm (ii)]   Suppose $\{\bm X_n;n\ge 1\}$ is a sequence of independent random vectors  in $(\Omega,\mathcal{H},\mbE)$, where $\bm X_n$ is $d_n$-dimensional,  $f_{n,j}\in C_{l,lip}(\mathbb R^{d_n})$ and $\sum_{n=1}^{\infty}\V(f_{n,j}(\bm X_n)\ge 1)=\infty$, $j=1,2,\ldots$,  then there exists $P\in\mathcal{P}$ such that
    \begin{equation}\label{eqlemBC2.3} P\left( \bigcap_{j=1}^{\infty}\big\{f_{n,j}(\bm X_n)\ge 1\;\; i.o.\big\} \right)=1.
   \end{equation}
\item[\rm (iii)] Suppose $\{\bm X_n;n\ge 1\}$ is a sequence of independent random vectors  in $(\Omega,\mathcal{H},\mbE)$, where $\bm X_n$ is $d_n$-dimensional. If $F_n$ is a $d_n$-dimensional close set with $\sum_{n=1}^{\infty} \mv(\bm X_n \not\in F_n)<\infty$, then there exists $P\in\mathcal{P}$ such that
          $$ P\left( \bm X_n\not\in F_n\; i.o.\right)=0; $$
          If $F_{n,j}$s are $d_n$-dimensional  closed sets with $\sum_{n=1}^{\infty} \V(\bm X_n\in F_{n,j})=\infty$, $j=1,2,\ldots$, then then there exists $P\in\mathcal{P}$ such that
          $$ P\left(\bigcap_{j=1}^{\infty}\big\{\bm X_n\in F_{n,j}\;\; i.o.\big\}\right)=1. $$
 \end{description}
\end{lemma}

\begin{proof}[\bf{Proof}] (i) and (ii) are special cases of (iii). But, to prove the general case (iii), we need to show the two special cases first.   Without loss of generality, we can assume $0\le X_n\le 2$, for otherwise, we can replace it by $0\vee X_n\wedge 2$. Write $\bm X=(X_1,X_2,\ldots)$.  Then $\mathcal{P}\bm X^{-1}$ is a compact family of probability measures under the weak convergence by the condition (CC). Choose a Lipschitz and non-increasing function $f_i$ such that $I\{x< -1\}\le f(x)\le I\{x< 0\}$. Then $f(x/\delta)\ge f(x/\epsilon)$ for $0<\delta<\epsilon$. In fact,  $f(x/\delta)= f(x/\epsilon)=0$ when $x\ge 0$, and when $x<0$, $x/\delta \le x/\epsilon$  and so $f(x/\delta)\ge f(x/\epsilon)$.

Now, we consider (i).  
Note $\mbe\left[f\left(\frac{X_n-1}{\delta}\right)\right]\le \mv(X_n<1)$ for all $\delta>0$. Let $A=\sum_{n=1}^{\infty} \mv(X_n<1)$. 
Then by the independence,
$$ \mbe\left[\sum_{n=1}^N f\left(\frac{X_n-1}{\delta}\right)\right]=
\sum_{n=1}^N \mbe\left[f\left(\frac{X_n-1}{\delta}\right)\right]\le A \text{ for } N \text{ and } \delta. $$
 Choose $\delta_N\searrow 0$. For each $N$, there exists $P_N\in\mathcal{P}$ such that
 $$P_N\left[\sum_{n=1}^N f\left(\frac{X_n-1}{\delta_N}\right)\right]\le   \mbe\left[\sum_{n=1}^N f\left(\frac{X_n-1}{\delta_N}\right)\right]+\frac{1}{N}\le A+\frac{1}{N}. $$
By the weak compactness, there exists a subsequence $N^{\prime}\nearrow \infty$ and a probability measure $P\in\mathcal{P}$ such that
$$ P_{N^{\prime}} \left[\sum_{n=1}^N f\left(\frac{X_n-1}{\delta_N}\right)\right]\to 
P  \left[\sum_{n=1}^N f\left(\frac{X_n-1}{\delta_N}\right)\right] \text{ for all } N. $$
Note $f(x/\delta_{N^{\prime}})\ge f(x/\delta_{N})$ for $N^{\prime}\ge N$. It follows that
$$P \left[\sum_{n=1}^L f\left(\frac{X_n-1}{\delta_N}\right)\right]
\le \liminf_{N^{\prime}\to \infty}P_{N^{\prime}} \left[\sum_{n=1}^{N^{\prime}} f\left(\frac{X_n-1}{\delta_{N^{\prime}}}\right)\right]\le A \text{ for all } L\le N. $$
Since $f((x-1)/\delta)\to I\{x<1\}$ as $\delta\to 0$.  By the continuity of the probability measure $P$, 
$$\sum_{n=1}^L  P(X_n<1)\le A \text{ for all } L\ge 1. $$
It follows that
$$ \sum_{n=1}^{\infty} P(X_n<1)\le A<\infty,  $$
which implies \eqref{eqlemBC2.1}.

Consider (ii).  Without loss of generality, assume $0\le f_{n,j}(\bm X_n)\le 2$.  By the independence, for each $m$ and $j$ we have
\begin{align*}
& \mbe\left[\prod_{i=m}^n f\left(\frac{f_{i,j}(\bm X_i)-1}{\epsilon}\right)\right]
=\prod_{i=m}^n \mbe\left[f\left(\frac{f_{i,j}(\bm X_i)-1}{\epsilon}\right)\right] \\
\le &\prod_{i=m}^n \mv(f_{i,j}(\bm X_i)<  1) = \prod_{i=m}^n \left(1-\V(f_{i,j}(\bm X_i)\ge   1) \right)
\le \exp\left\{-\sum_{i=m}^n  \V(f_{i,j}(\bm X_i)\ge   1)\right\}\; \text{ for all } \epsilon>0.  
\end{align*}
There exists $P_{\epsilon}\in \mathcal{P}$ such that
$$ P_{\epsilon}\left[\prod_{i=m}^n f\left(\frac{f_{i,j}(\bm X_i)-1}{\epsilon}\right)\right] 
\le \exp\left\{-\sum_{i=m}^n  \V(f_{i,j}(\bm X_i)\ge   1)\right\} +\epsilon. $$
By the weak compactness, there exists a sequence $\epsilon_N\searrow 0$ and a probability measure $P\in \mathcal{P}$ such that
\begin{align*}
& P\left[\prod_{i=m}^n f\left(\frac{f_{i,j}(\bm X_i)-1}{\epsilon_L}\right)\right] 
=\lim_{N\to \infty} P_{\epsilon_N}\left[\prod_{i=m}^n f\left(\frac{f_{i,j}(\bm X_i)-1}{\epsilon_L}\right)\right] \\
\le & \lim_{N\to \infty} P_{\epsilon_N}\left[\prod_{i=m}^n f\left(\frac{f_{i,j}(\bm X_i)-1}{\epsilon_N}\right)\right]\le \exp\left\{-\sum_{i=m}^n \V(f_{i,j}(\bm X_i)\ge   1)\right\}.
\end{align*}
Letting $L\to \infty$ yields 
$$\mv^{\mathcal{P}}\left(\bigcap_{i=m}^{n} \{f_{i,j}(\bm X_i)<  1\}\right)
\le P\left(\bigcap_{i=m}^{n} \{f_{i,j}(\bm X_i)<  1\}\right)\le \exp\left\{-\sum_{i=m}^n  \V(f_{i,j}(\bm X_i)\ge   1)\right\}. $$
Let $\delta_k=2^{-k}$. Note $\sum_{i=1}^{\infty} \V(f_{i,j}(\bm X_i)\ge   1)=\infty$. We choose the sequence $1=n_{0,0}<n_{1,1}<n_{2,1}<n_{2,2}<\ldots<n_{k,1}<\ldots<n_{k,k}<n_{k+1,1}<\ldots$ such that
$$ \mv^{\mathcal{P}}\left(\bigcap_{i=n_{k,j-1}+1}^{n_{k,j}} \{f_{i,j}(\bm X_i)<  1\}\right)\le \delta_{k+j},\;\; j\le k, k\ge 1,  $$
where $n_{k,0}=n_{k-1,k-1}$. Let $Z_{k,j}=\max_{n_{k,j-1}+1\le i\le n_{k,j}}f_{i,j}(\bm X_i)$. Then the random variables $Z_{1,1}$, $Z_{2,1}$, $Z_{2,2}$, $\ldots$, $Z_{k,1}$, $\ldots$, $Z_{k,k}$, $Z_{k+1,1},\ldots$ are independent under $\mbE$ with
$$ \mv^{\mathcal{P}}(Z_{k,j}<1)< \delta_{k+j}. $$
Note $\sum_{k=1}^{\infty}\sum_{j=1}^k\delta_{k+j}<\infty$. By (i), there exists a probability measure $P\in\mathcal{P}$ such that
$$ P\left(\bigcup_{l=1}^{\infty}\bigcap_{k=l}^{\infty}\bigcap_{j=1}^k\{Z_{k,j}\ge 1\}\right)=1. $$
On the event $\bigcup_{l=1}^{\infty}\bigcap_{k=l}^{\infty}\bigcap_{j=1}^k\{Z_{k,j}\ge 1\}$, there exists a $l_0$ such that $Z_{k,j}\ge 1$ for all $k\ge l_0$ and $1\le j\le k$. For each fixed $j$, when $k\ge j\vee l_0$ we have $Z_{k,j}\ge 1$, and hence $\{f_{n,j}(\bm X_n)\ge  1 \;\; i.o\}$ occurs.
It follows that
$$ \bigcup_{l=1}^{\infty}\bigcap_{k=l}^{\infty}\bigcap_{j=1}^k\{Z_{k,j}\ge 1\}\subset \bigcap_{j=1}^{\infty}\{f_{n,j}(\bm X_n)\ge  1 \;\; i.o\}. $$
\eqref{eqlemBC2.3} holds.

(iii) Denote $d(\bm x,F)=\inf\{\|\bm y-\bm x\|: \bm y\in F\}$. Then $d(\bm x,F)$ is a Lipschitz function of $\bm x$. If $F_{n,j}$ is a closed set, then
$$ \bm X_n\in F_{n,j}\Longleftrightarrow d(\bm X_n, F_{n,j})=0\Longleftrightarrow f_{n,j}(\bm X_n)=:1-1\wedge d(\bm X_n, F_{n,j})\ge 1. $$
The results follow  from (i) and (ii) immediately.
\end{proof}

The next lemma is Lemma 3.3 of Wittmann\cite{wit85}.
\begin{lemma}\label{wit}
Let $\{a_n\}$ be a sequence with $a_n\nearrow\infty$, then for any $\lambda>1$, there exists a sequence $\{n_k\}$ with $n_k\nearrow\infty$ such that
$$
	\lambda a_{n_k}\leq a_{n_{k+1}}\leq \lambda^3 a_{n_k+1}.
$$
\end{lemma}

The last lemma is about the properties for random variables with $C_{\hV}(|X|)<\infty$, which is Lemma 4.1 of Zhang\cite{zhang23b}.
\begin{lemma} \label{lemma3.5}
	Suppose $X\in\mathcal{H}$ with $C_{\hV}(|X|)<\infty$. Then
$$
	\sum_{i=1}^{\infty}\hV\left(|X|\geq Mi\right)<\infty,\enspace\text{ for all } M>0 \text{ or equivalently for some } M>0,
$$
and
$$
	\bE[(|X|-c)^+]=o(1),
$$
as $c\rightarrow\infty$.
\end{lemma}
The next lemma about $\bE$ is Proposition 1.1 of Zhang \cite{zhang23b}.
\begin{lemma}\label{lemma3.6} Consider a subspace of $\mathcal{H}$ as
\begin{equation}\label{eqlem3.6.1} \mathcal{H}_1=\big\{X\in\mathcal{H}: \lim_{c,d\to \infty}\mbE\big[(|X|\wedge d-c)^+\big]=0\big\}.
\end{equation}
Then for any $X\in\mathcal{H}_1$, $\bE[X]$ is well defined, and $(\Omega,\mathcal{H}_1,\bE)$ is a sub-linear expectation space.
\end{lemma}

\begin{lemma}\label{lemma3.7}  Suppose $X,Y\in \mathcal{H}_1$ are independent under the sub-liner expectation $\mbE$. Then
\begin{equation}\label{eqlem3.7.1} \bE[X+Y]=\bE[X]+\bE[Y].
\end{equation}
\end{lemma}

\begin{proof}[\bf{Proof}]
Note $|X^{(b)}-X|\le (|X|-b)^+$. It follows that
\begin{align*}
 \left|(X+Y)^{(b)}-(X^{(b)}+Y^{(b)})\right|\le & \left(\left|(X+Y)^{(b)}-(X+Y)\right|+|X^{(b)}-X|+|Y^{(b)}-Y|\right)\wedge (3b) \\
\le &  2 \cdot \Big((|X|-b/2)^++(|Y|-b/2)^+\Big)\wedge (3b).
 \end{align*}
 Hence,
 \begin{align*}
 \left|\mbE[(X+Y)^{(b)}]-\mbE[X^{(b)}+Y^{(b)}]\right|\le    2  \Big(\mbE[(|X|\wedge(3b)-b/2)^+]+\mbE[(|Y|\wedge(3b)-b/2)^+]\Big) \to 0,
 \end{align*}
 as $b\to \infty$. On the other hand, $\mbE[X^{(b)}+Y^{(b)}]=\mbE[X^{(b)}]+\mbE[Y^{(b)}] $ by the independence. Hence, \eqref{eqlem3.7.1} holds.
\end{proof}

\section{Proofs of the theorems}\label{sec:proof}
In this section, we prove the theorems in Section \ref{sec:result}.
\begin{proof}[\bf{Proof of Theorem \ref{th1}}]
	From (\ref{eqth1.1}) there exists a sequence $\epsilon_k\searrow0$ such that
$$
	\sum_{n=1}^\infty \frac{\mbE[X_n^2]}{\epsilon_n^2a_n^2}<\infty.
$$
By Lemma \ref{wit}, for all $\lambda>1$, there exists a sequence $n_k\nearrow\infty$ such that
\begin{equation}
	\lambda a_{n_k}\leq a_{n_{k+1}}\leq \lambda^3 a_{n_k+1}.\label{eq5.1}
\end{equation}
Then it follows from Kolmogorov's maximal inequality \eqref{kolV} that
$$
	\hV\left(\max_{n_k+1\leq n\leq n_{k+1}}(S_n-S_{n_k}-\mbE[S_n-S_{n_k}])\geq \epsilon_{n_k}a_{n_{k+1}}\right)\leq C\frac{\sum_{j=n_k+1}^{n_{k+1}}\mbE[X_j^2]}{\epsilon_{n_k}^2a_{n_{k+1}}^2}\leq C\sum_{j=n_k+1}^{n_{k+1}}\frac{\mbE[X_j^2]}{\epsilon_j^2a_j^2}.
$$
Hence
$$
\sum_{k=1}^\infty\hV\left(\max_{n_k+1\leq n\leq n_{k+1}}(S_n-S_{n_k}-\mbE[S_n-S_{n_k}])\geq \epsilon_{n_k}a_{n_{k+1}}\right)<\infty.
$$
By noting the countable sub-additivity of $\Vstar$ and Lemma \ref{lemBC1}, we have
$$
\Vstar\left(\max_{n_k+1\leq n\leq n_{k+1}}(S_n-S_{n_k}-\mbE[S_n-S_{n_k}])\geq \epsilon_{n_k}a_{n_{k+1}},i.o.\right)=0.
$$
Denote
$$
\Lambda_0=\left\{\max_{n_k+1\leq n\leq n_{k+1}}(S_n-S_{n_k}-\mbE[S_n-S_{n_k}])\geq \epsilon_{n_k}a_{n_{k+1}},i.o.\right\}^c.
$$
On $\Lambda_0$, there exists a positive integer $k_0$ such that for $k\geq k_0$, $\max_{n_k+1\leq n\leq n_{k+1}}(S_n-S_{n_k}-\mbE[S_n-S_{n_k}])\leq \epsilon_{n_k}a_{n_{k+1}}$. For $n\geq n_{k_0+1}$, there exitst a $k\geq k_0$ such that $n_k+1\leq n\leq n_{k+1}$, then from  \eqref{eq5.1},
\begin{align*}
	S_n-\mbE[S_n] =&S_n-S_{n_k}+\sum_{i=k_0+1}^{k}(S_{n_i}-S_{n_{i-1}}-\mbE[S_{n_i}-S_{n_{i-1}}])+S_{n_{k_0}}-\mbE[S_{n_{k_0}}]\\
 \leq&\epsilon_{n_k}a_{n_{k+1}}+\epsilon_{n_{k-1}}a_{n_k}+\cdots+\epsilon_{n_{k_0}-1}a_{n_{k_0}}+S_{n_{k_0}}-\mbE[S_{n_{k_0}}]\\
 \leq&\epsilon_{n_{k_0}-1}a_{n_{k+1}}(1+\frac{1}{\lambda}+\cdots+\frac{1}{\lambda^{k-k_0+1}})+S_{n_{k_0}}-\mbE[S_{n_{k_0}}]\\
 \leq&\epsilon_{n_{k_0}-1}a_{n_{k+1}}\frac{\lambda}{\lambda-1}+S_{n_{k_0}}-\mbE[S_{n_{k_0}}].
\end{align*}
And by \eqref{eq5.1} again,
\begin{align*}
\frac{S_n-\mbE[S_n]}{a_n} \leq&\epsilon_{n_{k_0}-1}\frac{a_{n_{k+1}}}{a_n}\cdot\frac{\lambda}{\lambda-1}+\frac{S_{n_{k_0}}-\mbE[S_{n_{k_0}}]}{a_n}\\
 \leq&\epsilon_{n_{k_0}-1}\frac{\lambda^4}{\lambda-1}+\frac{S_{n_{k_0}}-\mbE[S_{n_{k_0}}]}{a_n}.
\end{align*}
So
$$
\limsup_{n\rightarrow\infty}\frac{S_n-\mbE[S_n]}{a_n}\leq\epsilon_{n_{k_0}-1}\frac{\lambda^4}{\lambda-1}\rightarrow0,
$$
as $k_0\rightarrow\infty$, which means $\left\{\limsup_{n\rightarrow\infty}\frac{S_n-\mbE[S_n]}{a_n}>0\right\}\subset\Lambda_0^c$. It follows that
$$
\Vstar\left(\limsup_{n\rightarrow\infty}\frac{S_n-\mbE[S_n]}{a_n}>0\right)=0.
$$
On the other hand,
$$
\Vstar\left(\liminf_{n\rightarrow\infty}\frac{S_n-\mbe[S_n]}{a_n}<0\right)=\Vstar\left(\limsup_{n\rightarrow\infty}\frac{-S_n-\mbE[-S_n]}{a_n}>0\right)=0.
$$
The proof is completed.
\end{proof}

\par Having established Theorem \ref{th1}, we now turn our attention to proving the existence of some probability measure such that strong law of large numbers holds under the condition (CC).

\begin{proof}[\bf{Proof of Theorem \ref{th2}}]
We first show  \eqref{eqth2.1}. Denote $S_{m,n}=\sum_{i=m+1}^nX_i$ for $m<n$. By Kolmogorov's maximal inequality  \eqref{kolv}  and Kronecker lemma, we have
\begin{align*}	&  \mhV\left(\frac{S_{m,n}-\mbE[S_{m,n}]}{a_n}\leq-\epsilon\right)
= \mhV\left(\frac{-S_{m,n}-\mbe[-S_{m,n}]}{a_n}\geq\epsilon\right)\\
& \quad \leq   C\frac{\sum_{i=m+1}^n\mbE[X_i^2]}{\epsilon^2a_n^2}\leq C\frac{\sum_{i=1}^n\mbE[X_i^2]}{a_n^2}\rightarrow0.
\end{align*}
Similarly,
$$
	\mhV\left(\frac{-S_{m,n}-\mbE[-S_{m,n}]}{a_n}\leq-\epsilon\right)\rightarrow0.
$$
Set $\epsilon_k=2^{-k}$, then there exist $n_k\nearrow\infty$ such that $\frac{a_{n_{k+1}}}{a_{n_k}}\rightarrow\infty$, $\frac{\mbE[|S_{n_{k-1}}|]}{a_{n_k}}\to 0$ and
\begin{eqnarray*}
	\mhV\left(\frac{S_{n_{k-1},n_k}-\mbE[S_{n_{k-1},n_k}]}{a_{n_k}}<-\epsilon_k\right)\leq\epsilon_k,\\
	\mhV\left(\frac{-S_{n_{k-1},n_k}-\mbE[-S_{n_{k-1},n_k}]}{a_{n_k}}<-\epsilon_k\right)\leq\epsilon_k.
\end{eqnarray*}
Hence
\begin{eqnarray*}
	\sum_{k=1}^{\infty}\V\left(\frac{S_{n_{k-1},n_k}-\mbE[S_{n_{k-1},n_k}]}{a_{n_k}}\geq-\epsilon_k\right)=\infty,\\
	\sum_{k=1}^{\infty}\V\left(\frac{-S_{n_{k-1},n_k}-\mbE[-S_{n_{k-1},n_k}]}{a_{n_k}}\geq -\epsilon_k\right)=\infty.
\end{eqnarray*}
By Lemma \ref{lemBC2} (iii) and noting the independence,  there exists a probability measure $P\in \mathcal{P}$ such that
$$ P(\Omega_0)=1 $$
 where 
  $$ \Omega_0=\left\{ \frac{S_{n_{k-1},n_k}-\mbE[S_{n_{k-1},n_k}]}{a_{n_k}}\geq-\epsilon_k \; i.o.\right\}\bigcap \left\{\frac{-S_{n_{k-1},n_k}-\mbE[-S_{n_{k-1},n_k}]}{a_{n_k}}\geq -\epsilon_k\; i.o.\right\}. $$
  On the event $\Omega_0$, we have  
$$
	\limsup_{n\rightarrow\infty}\frac{S_{n_{k-1},n_k}-\mbE[S_{n_{k-1},n_k}]}{a_{n_k}}\geq 0 
$$
and
$$
	\liminf_{n\rightarrow\infty}\frac{S_{n_{k-1},n_k}-\mbe[S_{n_{k-1},n_k}]}{a_{n_k}}\leq0.
$$
Further, if $\limsup_{n\rightarrow\infty}\frac{S_n-\mbE[S_n]}{a_n}\leq0$ and $\liminf_{n\rightarrow\infty}\frac{S_n-\mbe[S_n]}{a_n}\geq0$, we have
\begin{align*}
	0\geq & \limsup_{n\rightarrow\infty}\frac{S_n-\mbE[S_n]}{a_n}\geq \limsup_{n\rightarrow\infty}\frac{S_{n_{k}}-\mbE[S_{n_{k}}]}{a_{n_k}}\\
\geq& \limsup_{k\rightarrow\infty}\frac{S_{n_{k-1},n_k}-\mbE[S_{n_{k-1},n_k}]}{a_{n_k}}
+\liminf_{k\rightarrow\infty}\frac{S_{n_{k-1}}-\mbe[S_{n_{k-1}}]}{a_{n_k}}+\liminf_{k\rightarrow\infty}\frac{ \mbe[S_{n_{k-1}}]-\mbE[S_{n_{k-1}}]}{a_{n_k}}\\
\geq& 0  ,\\
	0\leq &\liminf_{n\rightarrow\infty}\frac{S_n-\mbe[S_n]}{a_n}\leq\liminf_{n\rightarrow\infty}\frac{S_{n_k}-\mbe[S_{n_k}]}{a_{n_k}}\\
\leq &  \liminf_{n\rightarrow\infty}\frac{S_{n_{k-1},n_k}-\mbe[S_{n_{k-1},n_k}]}{a_{n_k}} +
+\limsup_{k\rightarrow\infty}\frac{S_{n_{k-1}}-\mbE[S_{n_{k-1}}]}{a_{n_k}}+\limsup_{k\rightarrow\infty}\frac{ \mbE[S_{n_{k-1}}]-\mbe[S_{n_{k-1}}]}{a_{n_k}}\\
\leq &0.
\end{align*}
Hence by Theorem \ref{th1}, it follows that
$$
	P\left(\limsup_{n\rightarrow\infty}\frac{S_n-\mbE[S_n]}{a_n}=0\enspace and\enspace\liminf_{n\rightarrow\infty}\frac{S_n-\mbe[S_n]}{a_n}=0\right)=1.
$$
\par Now we show  \eqref{eqth2.2}. Let $\{\epsilon_n\},\{n_k\}$ be the same as in Theorem \ref{th1}. By Lemma \ref{lemma2.7}, it follows that
$$
	\mhV\left(\max_{n_k+1\leq n\leq n_{k+1}}\left|\sum_{i=n_k+1}^n(X_i-\mu_i)\right|\geq\epsilon_{n_k}a_{n_{k+1}}\right)\leq C\frac{\sum_{j=n_k+1}^{n_{k+1}}\mbE[X_j^2]}{\epsilon_{n_k}^2a_{n_{k+1}}^2}\leq C \sum_{j=n_k+1}^{n_{k+1}}\frac{\mbE[X_j^2]}{\epsilon_j^2a_j^2}.
$$
Hence
$$
	 \sum_{k=1}^\infty \mhV\left(\max_{n_k+1\leq n\leq n_{k+1}}\left|\sum_{i=n_k+1}^n(X_i-\mu_i)\right|\geq\epsilon_{n_k}a_{n_{k+1}}\right)\leq C \sum_{j=1}^{\infty}\frac{\mbE[X_j^2]}{\epsilon_j^2a_j^2}<\infty.
$$
Note the independence, by Lemma \ref{lemBC2} (iii), there exists a probability measure $P\in\mathcal{P}$ such that
$$P\left(\max_{n_k+1\leq n\leq n_{k+1}}\left|\sum_{i=n_k+1}^n(X_i-\mu_i)\right|\geq\epsilon_{n_k}a_{n_{k+1}}\;\; i.o.\right)=0. $$
It is easily seen that on the event $\left\{\max_{n_k+1\leq n\leq n_{k+1}}\left|\sum_{i=n_k+1}^n(X_i-\mu_i)\right|\geq\epsilon_{n_k}a_{n_{k+1}}\;\; i.o.\right\}^c$, we have
$$
\lim_{n\rightarrow\infty}\frac{\sum_{i=1}^n(X_i-\mu_i)}{a_n}=0.
$$
Then the proof is completed.
\end{proof}

\par With the prelimiary results presented above, we now start to prove strong law of large numbers for $m$-dependent and stationary random variables as an application.

\begin{proof}[\bf{Proof of Theorem \ref{th3}}]
	We prove the theorem in five steps.

	\par \textbf{Step 1.} We show that, if $\{X_n;n\geq 1\}$ is a sequence of $m$-dependent and identically distributed random variables   with \eqref{eqth3.1}, then
\begin{equation}
	\Vstar\left(\limsup_{n\rightarrow\infty}\frac{S_n-\bE [S_n]}{n}>0\enspace or\enspace\liminf_{n\rightarrow\infty}\frac{S_n-\be [S_n]}{n}<0\right)=0.\label{eq4.7ad}
\end{equation}

 Let $Y_i=X_i^{(i)}$, then
\begin{align}
	\sum_{i=1}^\infty\frac{\mbE[Y_i^2]}{i^2}
 \leq&\sum_{i=1}^\infty\frac{C_{\hV}(Y_i^2)}{i^2}\leq\sum_{i=1}^\infty\frac{1}{i^2}\int_0^{i^2}\hV(X_1^2\geq y)\dif y=2\sum_{i=1}^\infty\frac{1}{i^2}\int_0^{i}y\hV(|X_1|\geq y)\dif y\nonumber\\
	 \leq& 2\sum_{i=1}^\infty\int_{i}^{i+1}\left(\frac{i+1}{i}\right)^2\frac{1}{x^2}\int_0^{x}y\hV(|X_1|\geq y)\dif y\dif x\nonumber\\
	 \leq& 8\int_1^\infty\frac{\dif x}{x^2}\int_0^{x}y\hV(|X_1|\geq y)\dif y
	 \leq  8\int_{0}^\infty y\hV(|X_1|\geq y)\dif y\int_y^\infty\frac{\dif x}{x^2}\nonumber\\
	 \leq& 8\int_{0}^\infty \hV(|X_1|\geq y)\dif y=8 C_{\hV}(|X_1|)<\infty.
\end{align}
There exists a sequence $M_i\nearrow\infty$ such that $\sum_{i=1}^\infty\frac{M_i\mbE[Y_i^2]}{i^2}<\infty$. Next, we will use the technique of segmenting summation to piecewise summations of  large and small blocks.  Define $a_0=0,a_n=a_{n-1}+l_n$, where $l_n$ is chosen such that $l_n\nearrow\infty$
\begin{equation}
	l_n\leq M_{a_{n-1}+1}^{1/4},\enspace l_n\leq n^{1/4}.\label{l}
\end{equation}
Denote
$$
	Z_n=\sum_{i=a_{n-1}+1}^{a_n-m}Y_i,\enspace W_n=\sum_{i=a_{n}-m+1}^{a_n}Y_i.
$$
Then $\{Z_n;n\ge 1\}$ and $\{W_n;n\ge 1\}$ are    sequences of independent random variables, respectively. The former is the main part we want to investigate, while the latter we need to prove that it can be ignored.  Noting the construction of $\{l_n\}$, we have
\begin{equation}\label{eqZ}
	\sum_{n=1}^\infty\frac{\mbE[Z_n^2]}{a_n^2}\leq \sum_{n=1}^\infty\frac{(l_n-m)\sum_{i=a_{n-1}+1}^{a_n}\mbE[Y_i^2]}{a_n^2}\leq\sum_{n=1}^\infty\sum_{i=a_{n-1}+1}^{a_n-m}\frac{M_i\mbE[Y_i^2]}{i^2}
=\sum_{i=1}^\infty\frac{M_i\mbE[Y_i^2]}{i^2}<\infty.
\end{equation}
Similarly,
\begin{equation}\label{eqW}
	\sum_{n=1}^\infty\frac{\mbE[W_n^2]}{a_n^2}\leq \sum_{n=1}^\infty\frac{m\sum_{i=a_{n}-m+1}^{a_n}\mbE[Y_i^2]}{a_n^2}\leq m\sum_{i=1}^\infty\frac{\mbE[Y_i^2]}{i^2}<\infty.
\end{equation}
By  \eqref{eqW}  and Theorem \ref{th1}, it follows that
\begin{equation}\label{eqlimitW}
	\Vstar\left(\limsup_{n\rightarrow\infty}\frac{\sum_{i=1}^nW_i-\mbE\left[\sum_{i=1}^nW_i\right]}{a_n}>0\enspace or\enspace\liminf_{n\rightarrow\infty}\frac{\sum_{i=1}^nW_i-\mbe[\sum_{i=1}^nW_i]}{a_n}<0\right)=0.
\end{equation}
By  \eqref{eqZ}  and Theorem \ref{th1}, we have
\begin{equation}\label{eqlimitZ}
	\Vstar\left(\limsup_{n\rightarrow\infty}\frac{\sum_{i=1}^nZ_i-\mbE\left[\sum_{i=1}^nZ_i\right]}{a_n}>0\enspace or\enspace\liminf_{n\rightarrow\infty}\frac{\sum_{i=1}^nZ_i-\mbe[\sum_{i=1}^nZ_i]}{a_n}<0\right)=0.
\end{equation}
Hence, for \eqref{eq4.7ad} it is sufficient to  that
\begin{align}
  \Vstar & \left(\lim_{n\rightarrow\infty}\max_{a_n+1\leq N\leq a_{n+1}}\left|\frac{S_N}{N}-\frac{\sum_{i=1}^n Z_i}{a_n}\right|\neq0\right)=0,
\label{eq5.10ad}\\
& \;\; \bE\left[\max_{a_n+1\leq N\leq a_{n+1}}\left|\frac{S_N}{N}-\frac{\sum_{i=1}^n Z_i}{a_n}\right|\right]\to 0.
\label{eq5.11ad}
\end{align}

Since $\mbE[|Y_i|]\leq \bE[|X_1|]$ for any $i$, we have
\begin{equation}\label{eqlimitmeanW}
	\frac{\mbE\left|\left[\sum_{i=1}^nW_i\right]\right|}{a_n}\leq \frac{\mbE\left[\sum_{i=1}^n|W_i|\right]}{a_n} \leq\frac{nm\bE[|X_1|]}{a_n}\rightarrow0.
\end{equation}
Let $T_n=\sum_{i=1}^nY_i$. By \eqref{eqlimitW} and \eqref{eqlimitmeanW} it follows that
\begin{equation} \Vstar\left(\lim_{n\rightarrow\infty}\left(\frac{T_{a_n}}{a_n}-\frac{\sum_{i=1}^nZ_i}{a_n}\right)\neq0\right)
=\Vstar\left(\lim_{n\rightarrow\infty}\frac{\sum_{i=1}^nW_i}{a_n}\neq0\right)=0,\label{eq5.7}
\end{equation}
and
\begin{equation}
	\lim_{n\rightarrow\infty}\frac{\mbE[\left|T_{a_n} -\sum_{i=1}^n Z_i\right|]}{a_n}=0.\label{eq5.8}
\end{equation}

For $a_n+1\leq N\leq a_{n+1}$, we have
\begin{align}
	\left|\frac{T_N}{N}-\frac{T_{a_n}}{a_n}\right|
=&\left|\frac{T_N-T_{a_n}}{N}+T_{a_n}\cdot\frac{a_n-N}{Na_n}\right|\nonumber\\
 \leq&\frac{\sum_{i=a_n+1}^{a_{n+1}}|Y_i|}{a_n}+\frac{|T_{a_n}|}{a_n}\cdot\frac{a_{n+1}-a_n}{a_n}.\label{eq5.12}
\end{align}
Note that
$$
	\sum_{i=1}^\infty \Vstar\left(|Y_i|>\frac{i}{M_i^{1/2}}\right)\leq\sum_{i=1}^\infty\frac{M_i\mbE[Y_i^2]}{i^2}<\infty,
$$
it follows from Lemma \ref{BC} that
$$
	\Vstar\left(|Y_i|>\frac{i}{M_i^{1/2}},i.o.\right)=0.
$$
On the event $\left\{|Y_i|>\frac{i}{M_i^{1/2}},i.o.\right\}^c$, for $n$ large enough,
\begin{align}
	\frac{\sum_{i=a_n+1}^{a_{n+1}}|Y_i|}{a_n}
 \leq&\frac{\sum_{i=a_n+1}^{a_{n+1}}i/(M_i^{1/2})}{a_n}\leq \frac{1}{a_nM_{a_n+1}^{1/2}}\sum_{i=a_n+1}^{a_{n+1}}i\nonumber\\
\leq& \frac{a_{n+1}}{a_n}(a_{n+1}-a_n)\frac{1}{M_{a_n+1}^{1/2}}=\left(1+\frac{l_{n+1}}{a_n}\right)l_{n+1}\frac{1}{M_{a_n+1}^{1/2}}\nonumber\\
 \leq&\left(1+\frac{l_{n+1}}{a_n}\right)\frac{1}{M_{a_n+1}^{1/4}}\rightarrow0.\label{eq5.13}
\end{align}
Note that $\frac{a_{n+1}-a_n}{a_n}=\frac{l_{n+1}}{a_n}\rightarrow 0$. We conclude from  \eqref{eq5.12}  - \eqref{eq5.13}  that
\begin{equation}
\Vstar\left(\lim_{n\rightarrow\infty}\max_{a_n+1\leq N\leq a_{n+1}}\left|\frac{T_N}{N}-\frac{T_{a_n}}{a_n}\right|\neq0\right)=0.\label{eq5.14}
\end{equation}
Further, by Lemma \ref{lemma3.5},
$$
	\sum_{i=1}^\infty\hV(X_i\neq Y_i)\leq\sum_{i=1}^\infty\hV(|X_1|>i/2)<\infty.
$$
Hence
\begin{equation}
	\Vstar(X_i\neq Y_i,i.o.)=0.\label{eq5.16}
\end{equation}
By  \eqref{eq5.7},  \eqref{eq5.14}  and \eqref{eq5.16},  \eqref{eq5.10ad}  is proved.

On the other hand,  by Lemma \ref{lemma3.5},
\begin{equation}
\frac{\bE\left[\left|T_{a_n}-S_{a_n}\right|\right]}{a_n}\leq\frac{\sum_{i=1}^{a_n}\bE[(|X_1|-i)^+]}{a_n}\rightarrow 0.\label{eq5.12ad}
\end{equation}
 Similar to \eqref{eq5.12}, we have
\begin{align}
	\bE\left[\max_{a_n+1\leq N\leq a_{n+1}}\left|\frac{S_N}{N}-\frac{S_{a_n}}{a_n}\right|\right]
 \le&\frac{\sum_{i=a_n+1}^{a_{n+1}}\bE[|X_i|]}{a_n}+\frac{\bE[|T_{a_n}|]}{a_n}\cdot\frac{a_{n+1}-a_n}{a_n}\nonumber \\
 \le & 2\bE[|X_1]\frac{l_{n+1}}{a_n}\to 0.\label{eq5.13ad}
\end{align}

Combine  \eqref{eq5.8},  \eqref{eq5.12ad}  and  \eqref{eq5.13ad} yields  \eqref{eq5.11ad}.

\par \textbf{Step 2.} Suppose $\mbE$ satisfies the condition (CC).   We  show that, if $\{X_n;n\geq 1\}$ is a sequence of $m$-dependent and identically distributed random variables   with \eqref{eqth3.1}, then
\begin{equation}
	P\left(\limsup_{n\rightarrow\infty}\frac{S_n-\bE[S_n]}{n}=0\enspace and\enspace\liminf_{n\rightarrow\infty}\frac{S_n-\be[S_n]}{n}=0\right)=1,
\label{eq4.8ad}
\end{equation}

 By  \eqref{eqZ}  and Theorem \ref{th2}(i),   there exists $P\in\mathcal{P}$ such that
$$
	P\left(\limsup_{n\rightarrow\infty}\frac{\sum_{i=1}^n Z_i-\mbE[\sum_{i=1}^n Z_i]}{a_n}=0\enspace and\enspace\liminf_{n\rightarrow\infty}\frac{\sum_{i=1}^n Z_i-\mbe[\sum_{i=1}^n Z_i]}{a_n}=0\right)=1,
$$
which, together with \eqref{eq5.10ad} and \eqref{eq5.11ad}, yields \eqref{eq4.8ad}.

	\par\textbf{Step 3.} We show that, if $\{X_n;n\geq 1\}$ is a sequence of $m$-dependent and linear stationary random variables  with \eqref{eqth3.1},
 then the limits in  \eqref{eqth3.2}  exist.

 Note that if $C_{\hV}(|X|)<\infty$, then $\bE[X]$ and $\be[X]$ are well-defined.  Note that
$$
	\be[X_1]\leq\frac{\sum_{i=1}^n\be[X_i]}{n}\leq\frac{\be[S_n]}{n}\leq\frac{\bE[S_n]}{n}\leq\frac{\sum_{i=1}^n\bE[X_i]}{n}=\bE[X_1],
$$
by Lemma \ref{lemma3.6}. So, there exists a sequence $k_n\nearrow\infty$ such that the limit of $\frac{\bE[S_{k_n}]}{k_n}$ exists as $n\rightarrow\infty$ and we denote the limit by $\um$. Without loss of generality, we can assume that $k_n=o(\sqrt{n})$, otherwise we can let $l_n=\sup\{i:k_i\leq n^{1/4}\}$ and replace $k_n$ with $k_{l_n}$. Now let $m_n=\left[\frac{n}{k_n}\right]$. Then
\begin{align*}
	S_n=&\sum_{i=1}^{m_n}\sum_{j=(i-1)k_n+1}^{ik_n}X_j+S_n-S_{m_nk_n}\\
		=&\sum_{i=1}^{m_n}(S_{ik_n}-S_{(i-1)k_n})+S_n-S_{m_nk_n}\\
		=&\sum_{i=1}^{m_n}(S_{ik_n-m}-S_{(i-1)k_n})+S_n-S_{m_nk_n}+\sum_{i=1}^{m_n}(S_{ik_n}-S_{ik_n-m}).
\end{align*}
Since
\begin{eqnarray*}
	\frac{|\bE[S_n-S_{m_nk_n}]|}{n}\leq\frac{n-m_nk_n}{n}\bE[|X_1|]\rightarrow 0,\\
	\frac{\left|\bE\left[\sum_{i=1}^{m_n}(S_{ik_n}-S_{ik_n-m})\right]\right|}{n}\leq m\frac{m_n}{n}\bE[|X_1|]\rightarrow0,
\end{eqnarray*}
and   $S_{ik_n-m}-S_{(i-1)k_n},i=1,\cdots,m_n$, are independent and identically distributed by the $m$-independence and linear stationary, by Lemma \ref{lemma3.7} we have
\begin{align*}
\frac{\bE[S_n]}{n} =&\frac{\sum_{i=1}^{m_n}\bE[S_{ik_n-m}-S_{(i-1)k_n}]}{n}+o(1)
		=\frac{m_n\bE[S_{k_n-m}]}{n}+o(1)\\
		 =&\frac{m_n\bE[S_{k_n}]}{n}+o(1)
		=\frac{m_nk_n}{n}\frac{\bE[S_{k_n}]}{k_n}+o(1)\rightarrow\um.
\end{align*}  Similarly, $\lm=\lim_{n\rightarrow\infty}\frac{\be[S_n]}{n}$ exists. And the inequalities $\be[X_1]\leq\lm\leq\um\leq\bE[X_1]$ are obvious.

\par \textbf{Step 4.} We show \eqref{eqth3.5}.

Suppose $\mbE$ satisfies the condition (CC).
Let $\underline{\mu}\le \mu_j\le \overline{\mu}$, $x_i=\mbe[Z_i]\vee(\sum_{j=a_{i-1}+1}^{a_{i}}\mu_j)\wedge \mbE[Z_i]$.
By  \eqref{eqZ}  and Theorem \ref{th2}(ii), we have there exists $P\in\mathcal{P}$ such that
$$
	P\left(\lim_{n\rightarrow\infty}\frac{\sum_{i=1}^n (Z_i-x_i)}{a_n}=0\right)=1.
$$
It is easily seen that
$$ \lim_{i\to \infty}\frac{\mbE[Z_i]}{l_i}=\lim_{i\to \infty}\frac{\mbE\left[\sum_{j=a_{i}+1}^{a_{i+1}}Y_i\right]}{l_i}
=\lim_{i\to \infty}\frac{\bE\left[\sum_{j=a_{i}+1}^{a_{i+1}}X_i\right]}{l_i}=\lim_{i\to \infty}\frac{\bE[S_{l_i}]}{l_i}=\overline{\mu},$$
$$ \lim_{i\to \infty}\frac{\mbe[Z_i]}{l_i}=\lim_{i\to \infty}\frac{\mbe\left[\sum_{j=a_{i}+1}^{a_{i+1}}Y_i\right]}{l_i}
=\lim_{i\to \infty}\frac{\be\left[\sum_{j=a_{i}+1}^{a_{i+1}}X_i\right]}{l_i}=\lim_{i\to \infty}\frac{\be[S_{l_i}]}{l_i}=\underline{\mu}$$
and
$$|x_i-\sum_{j=a_{i-1}+1}^{a_i}\mu_j|\le \big|\mbE[Z_i]-l_i\overline{\mu}\big|+\big|\mbe[Z_i]-l_i\underline{\mu}\big|. $$
It follows that
$$ \frac{|\sum_{i=1}^n x_i-\sum_{j=1}^{a_n}\mu_j|}{a_n}\le \frac{\sum_{i=1}^n l_i\Big( \big|\frac{\mbE[Z_i]}{l_i}-\overline{\mu}\big|
+\big|\frac{\mbe[Z_i]}{l_i}-\underline{\mu}\big|\Big)}{a_n}\to 0. $$
Hence
$$
	P\left(\lim_{n\rightarrow\infty}\frac{\sum_{i=1}^n Z_i-\sum_{j=1}^{a_n}\mu_j}{a_n}=0\right)=1.
$$
On the other hand, it is easily seen that
$$ \max_{a_n+1\le N\le a_{n+1}}\left|\frac{\sum_{j=1}^N \mu_j}{N}-\frac{\sum_{j=1}^{a_n}\mu_j}{a_n}\right|\to 0. $$
By \eqref{eq5.14}, \eqref{eqth3.5} is proved.

\par \textbf{Step 5.} At last, we show   \eqref{eqth3.6}.

When $\bE[X_1]=\be[X_1]$, then $\um=\lm$ and \eqref{eqth3.6} follows from \eqref{eqth3.5} immediately. When $\bE[X_1]>\be[X_1]$, there exists $c$ such
that $\alpha=: \mbe[(-c)\wedge X_1\vee c]< \mbE[(-c)\wedge X_1\vee c]:=\beta$. Let $Y_i=(-c)\wedge X_{(m+1)i}\vee c$. Then $\{Y_i;i\ge 1\}$ is a sequence of independent and identically distributed bounded random variables with $\mbE[Y_i]=\beta$ and $\mbe[Y_i]=\alpha$. By \eqref{eqth3.4}, there exists a probability measure $P\in \mathcal{P}$ such that
$$
	P\left(\limsup_{n\rightarrow\infty}\frac{\sum_{i=1}^n Y_i}{n}=\beta \text{ and } \liminf_{n\rightarrow\infty}\frac{\sum_{i=1}^n Y_i}{n}=\alpha\right)=1.
$$
On the other hand, $\{Y_i-E_P[Y_i|\mathcal{F}_{i-1}];i\ge 1\}$ is a sequence of bounded martingale differences under $P$, where $\mathcal{F}_i=\sigma(X_{(m+1)j};j=1,\ldots,i)$. By the law of large numbers of martingales,
$$
	P\left(\lim_{n\rightarrow\infty}\frac{\sum_{i=1}^n (Y_i-E_P[Y_i|\mathcal F_{i-1}])}{n}=0\right)=1.
$$
It follows that
$$
	P\left(\limsup_{n\rightarrow\infty}\frac{\sum_{i=1}^n E_P[Y_i|\mathcal F_{i-1}]}{n}=\beta \text{ and } \liminf_{n\rightarrow\infty}\frac{\sum_{i=1}^n E_P[Y_i|\mathcal F_{i-1}]}{n}=\alpha\right)=1.
$$
Note
$$ \alpha=\mbe[Y_i]\le E_P[Y_i|\mathcal F_{i-1}]\le \mbE[Y_i]=\beta \;\; a.s. \text{ under }P, $$
c.f. Guo, Li and Li \cite{GLL21}.
Hence, there exists a sequence of real numbers $\{y_i;i\ge 1\}$ with $\alpha\le y_i\le \beta$ such that
$$ \limsup_{n\rightarrow\infty}\frac{\sum_{i=1}^n y_i}{n}=\beta \text{ and } \liminf_{n\rightarrow\infty}\frac{\sum_{i=1}^ny_i}{n}=\alpha. $$
Let
$$\mu_i=\frac{y_i-\alpha}{\beta-\alpha}(b-a)+a. $$
Then $\mu_i\in [a,b]$ and
$$ \limsup_{n\rightarrow\infty}\frac{\sum_{i=1}^n \mu_i}{n}=b \text{ and } \liminf_{n\rightarrow\infty}\frac{\sum_{i=1}^n\mu_i}{n}=a, $$
which implies
$$ C\left\{\frac{\sum_{i=1}^n \mu_i}{n}\right\}=[a,b]. $$
\eqref{eqth3.6} follows from \eqref{eqth3.5} immediately.
\end{proof}

\begin{proof}[\bf{Proof of Theorem \ref{th4}}] By Lemma \ref{lemma3.5},
$$
	\sum_{i=1}^{\infty}\hV\left(|X_1|\geq Mi\right)=\infty,\enspace\forall M>0.
$$
There exists $l_i\nearrow \infty$ such that
$$
	\sum_{i=1}^{\infty}\hV\left(|X_1|\geq 2 l_i i \right)=\infty.
$$
By \eqref{eq1.2}, 
$$
	\sum_{i=1}^{\infty}\hV\left(|X_{i(m+1)}|\geq   l_i i \right)=\infty.
$$
By the $m$-dependence, $\{X_{i(m+1)};i\ge 1\}$ are independent under $\mbE$. Hence, by Lemma \ref{lemBC2} (iii) there exists a probability measure $P\in \mathcal{P}$ such that
$$ P\left(|X_{i(m+1)}|\ge l_ii \;\; i.o.\right)=1. $$
On the other hand, on the event $\{|X_{i(m+1)}|\ge l_ii \;\; i.o.\}$, we have
$$ \infty=\limsup_{i\to \infty}\frac{|  X_{i(m+1)}|}{(m+1)i}\le \limsup_{n\to \infty}\frac{|  X_n|}{n}=\limsup_{n\to \infty}\frac{|  S_n-S_{n-1}|}{n}\le 2\limsup_{n\to \infty}\frac{|  S_n|}{n}. $$
\eqref{eqth4.2} is proved.
\end{proof}
\bibliographystyle{plain}

\end{document}